\documentclass[11pt]{amsart}

\usepackage{color}
\usepackage{mathrsfs}
\usepackage{url}

\usepackage[arrow, matrix, curve]{xy}

\usepackage[american]{babel}
\usepackage[latin1]{inputenc}

\usepackage[neveradjust]{paralist}               	
\usepackage[parfill]{parskip}			
\usepackage{a4wide}
\usepackage{amsmath,amsthm}	
\usepackage{amssymb}
\usepackage[T1]{fontenc}

\usepackage[bottom]{footmisc}			
\usepackage{bbm, dsfont}  
\renewcommand{\1}{\mathbbm{1}}
\newcommand{\imp}{\Rightarrow}
\newcommand{\en}{\enspace}
\newcommand{\sk}[1]{\left\langle #1 \right\rangle}

\def \A {\mathcal{A}}
\def \B {\mathcal{B}}
\def \L {\mathscr{L}}

\def \R {\mathbb{R}}

\def \N {\mathbb{N}}
\def \C {\mathbb{C}}
\def \l {\lambda}
\def \e {\varepsilon}

\def \S {\mathcal{S}}

\def \ol {\overline}

\def \a {\alpha}
\def \b {\beta}

\def \ph {\varphi}

\def \Torus {\mathbb{T}}

\def \rg {\operatorname{rg}}
\def \lin {\operatorname{lin}}

\def \co {\operatorname{co}}

\def \Fix {\operatorname{Fix}}
\def \ex {\operatorname{ex}}

\theoremstyle{plain}
\newtheorem*{thm*}{Theorem}
\newtheorem{thm}{Theorem}[section]
\newtheorem{lemma}[thm]{Lemma}
\newtheorem{prop}[thm]{Proposition}
\newtheorem{cor}[thm]{Corollary}
\theoremstyle{definition}
\newtheorem{definition}[thm]{Definition}
\newtheorem{example}[thm]{Example}
\newtheorem{examples}[thm]{Examples}
\newtheorem{remark}[thm]{Remark}

\title{Uniform families of ergodic operator nets}
\author{Marco Schreiber}
\address{Institute of Mathematics\\University of T\"ubingen\\Auf der
  Morgenstelle 10\\72076 T\"ubingen\\Germany}
\email{masc@fa.uni-tuebingen.de}
\date\today

\keywords{amenable semigroups, mean ergodic theorem, topological
  Wiener-Wintner theorems}

\begin{document}		

\maketitle

\begin{abstract}
 We study mean ergodicity in amenable operator semigroups and
 establish the connection to the convergence of strong and
 weak ergodic nets. We then use these results in order to show the convergence of uniform families of ergodic nets that appear in
 topological Wiener-Wintner theorems.
\end{abstract}

The classical mean ergodic theorem (see \cite[Chapter 2.1]{krengel85}) is
concerned with the convergence of the Cesàro means 
$\frac{1}{N}\sum_{n=0}^{N-1}S^n$ for some power bounded operator
$S$ on a Banach space $X$. The natural extension of the Cesàro means
for representations $\S$ of general semigroups is the notion of an
\emph{ergodic net} as introduced by Eberlein \cite{eberlein49} and
Sato \cite{sato78}. In the first part of this paper we discuss and
slightly modify this concept in order to adapt it better for the
study of operator semigroups. Sato showed in \cite{sato78} that in
amenable semigroups there
always exist weak ergodic nets. We extend this result and show
that even strong ergodic nets exist. Using this
fact we then state a mean ergodic theorem connecting the convergence of strong and weak
ergodic nets and the existence of a zero element in the closed convex
hull of $\S$.

In the second part we develop the adequate framework for
investigating uniform convergence in so-called topological Wiener-Wintner
theorems. 
In the simplest situation these theorems deal with the
convergence of averages 
$$\textstyle\frac{1}{N}\sum_{n=0}^{N-1} \lambda^n S^n $$ 
 for some operator $S$ on spaces $C(K)$ and $\lambda$ in the unit circle
$\Torus$.
Assani \cite{assani03} and Robinson \cite{robinson94} asked when 
this convergence is uniform in $\lambda\in\Torus$. 
Subsequently, their results have been
generalised in different ways by Walters
\cite{walters96}, 
Santos and Walkden 
\cite{santos07}
and Lenz 
\cite{lenz09a, lenz09}.
We propose and study \emph{uniform families of ergodic nets} as an
appropriate concept for unifying and generalizing these and other
results.

\section{Amenable and mean ergodic operator semigroups}

We start from a semitopological semigroup $G$ and refer to Berglund et
al. \cite[Chapter 1.3]{berglund89} for an introduction to this theory.
Let $X$ be a Banach space and denote by $\L(X)$ the set of bounded
linear operators on $X$. We further assume that $\S=\{S_g:g\in G\}$ is
a bounded representation of $G$ on $X$, i.e.,
\begin{enumerate}[(i)]
\item $S_g\in \L(X)$ for all $g\in G$ and $\sup_{g\in G}\|S_g\|<\infty$,
\item $S_g S_h=S_{h g}$ for all $g,h\in G$,
\item $g\mapsto S_g x$ is continuous for all $x\in X$.
\end{enumerate}

For a bounded representation $\S$ we denote by $\co\S$
its convex hull and by $\ol{\co}\S$ the closure with
respect to the strong operator topology. Notice that $\S$ as well as $\co\S$ and
$\ol{\co}\S$ are topological semigroups with respect to the strong
and semitopological semigroups with respect to
the weak operator topology.

A \emph{mean} on
the space $C_b(G)$ of bounded continuous functions on $G$ is a linear
functional $m\in C_b(G)'$ satisfying
 $\sk{m,\1}=\|m\|=1$. 

A mean $m\in C_b(G)'$ is called \emph{right (left) invariant} if
$$\sk{m,R_g f}=\sk{m,f} \quad \left(\sk{m,L_g f}=\sk{m,f}\right)
 \quad\forall g\in G, f\in C_b(G),$$ 
where 
$R_g f(h)=f(hg)$ and $L_gf(h)=f(gh)$ for $h\in G$.

A mean $m\in C_b(G)'$ is called \emph{invariant} if it is both right and left invariant.

The semigroup $G$ is called \emph{right (left) amenable} if there
exists a right (left) invariant mean on $C_b(G)$. It is called
\emph{amenable} if there exists an invariant mean on $C_b(G)$ (see Berglund et
al. \cite[Chapter 2.3]{berglund89} or the survey article of Day
\cite{day69}).

For simplicity, we shall restrict ourselves to right
amenable and amenable semigroups, although most of the results also
hold for left amenable semigroups.

Notice that if $\S:=\{S_g:g\in G\}$ is a bounded representation of a
(right) amenable semigroup $G$
on $X$, then $\S$ endowed with the strong as well as the weak
operator topology is also (right) amenable. Indeed, if $\tilde{m}\in
C_b(G)'$ is a (right) invariant mean on $C_b(G)$, then $m\in C_b(\S)'$ given
by 
$$\sk{m,f}:=\langle\tilde{m},\tilde{f}\rangle\quad (f\in C_b(\S))$$
defines a (right) invariant mean on $C_b(\S)$, where $\tilde{f}(g)=f(S_g)$. 

In the following, the space $\L(X)$ will be endowed with
the strong operator topology unless stated otherwise.
\begin{definition}
\label{def:ergodic-nets}
  A net $(A_\a^\S )_{\a\in \A}$ of operators in $\L(X)$ is called a
  \emph{strong right (left) $\S$-ergodic net}  if the following conditions hold.
  \begin{enumerate}
  \item $A_\a^\S \in \ol{\co}\S$ for all $\a\in \A$.

  \item $(A_\a^\S)$ is \emph{strongly right (left) asymptotically invariant}, i.e.,

$\lim_\a A_\a^\S x-A_\a^\S S_g x=0\en \left(\lim_\a A_\a^\S x-
S_g A_\a^\S x=0\right)$ for all $x\in X$ and  $g\in G$.
  \end{enumerate}
The net $(A_\a^\S )$ is called a
  \emph{weak right (left) $\S$-ergodic net} if $(A_\a^\S)$ is
  \emph{weakly right (left) asymptotically invariant}, i.e., if the
  limit in (2) is taken 
  with respect to the weak topology $\sigma(X,X')$ on $X$.
The net $(A_\a^\S )$ is called a \emph{strong (weak) $\S$-ergodic net} if it
is a strong (weak) right and left $\S$-ergodic net. 
\end{definition}

We note that our definition differs slightly from
that of Eberlein~\cite{eberlein49}, Sato~\cite{sato78} and
Krengel~\cite[Chapter 2.2]{krengel85}. Instead of condition (1) they
require only

\begin{enumerate}[(1')]
\item 
$A_\a^\S x \in \ol{\co}\S x$ for all $\a\in \A$ and $x\in X$.
\end{enumerate}

However, the
existence of (even strong) ergodic nets in the sense of
Definition~\ref{def:ergodic-nets} is ensured by
Corollary~\ref{cor:existence-of-strong-ergodic-nets}. 
Moreover, both definitions lead to the same convergence results (see
Theorem~\ref{thm:mean-ergodic-everywhere} below). 
The reason is that if the limit $Px:=\lim_\a A_\a^\S x$ exists for all $x\in X$, then
the operator $P$ satisfies $P\in\ol{\co}\S$ rather than only
$Px\in\ol{\co}\S x$ for all $x\in X$ (see Nagel \cite[Theorem~1.2]{nagel73}). 


Here are some typical examples of ergodic nets. 

\begin{examples}
\label{ex:cesaro-nets}
  \begin{enumerate}[(a)]
  
\item          
Let $S\in \L(X)$ with $\|S\|\le 1$ and consider the
  representation $\S=\{S^n:n\in\N\}$ of the semigroup $(\N,+)$ on
  $X$. Then the \emph{Cesàro means} $(A_N^\S)_{N\in\N}$ given by
$$A_N^\S :=\frac{1}{N}\sum_{n=0}^{N-1} S^n$$ 
form a strong $\S$-ergodic net.

\item           
In the situation of (a), the \emph{Abel means} $(A_r^\S)_{0<r<1}$ given by
$$A_r^\S:=(1-r)\sum_{n=0}^\infty r^n S^n$$
form a strong $\S$-ergodic net.

\item          
Consider the semigroup $(\R_+,+)$ being represented on $X$ by a bounded
  $C_0$-semigroup $\S=\{S(t): t\in\R_+\}$. Then
  $(A_s^\S)_{s\in\R_+}$ given by
$$A_s^\S x:=\frac{1}{s}\int_0^s S(t)x\,dt\quad (x\in X)$$
is a strong $\S$-ergodic net.

\item            
\label{ex:eberlein}
Let $\S=\{S_g:g\in G\}$ be a bounded representation on $X$ of an abelian
semigroup $G$. Order the elements of $\co\S$ by
  setting $U\le V$ if there exists $W\in \co\S$ such that
  $V=WU$. Then $(A_U^S)_{U\in\co\S}$ given by 
$$A_U^\S:=U$$ 
is a strong $\S$-ergodic net.

\item           
\label{ex:folner}
Let $H$ be a locally compact group with left Haar measure
  $|\cdot|$ and let $G\subset H$ be a subsemigroup. Suppose that there
  exists a \emph{F\o lner net} $(F_\a)_{\a\in \A}$ in $G$ 
(see \cite[Chapter 4]{paterson88}), i.e., a net
  of compact sets such that
  $|F_\a|>0$ for all $\a\in\A$ and
$$\lim_\a \frac{|g F_\a \Delta F_\a|}{|F_\a|} =0 \quad\forall g\in G,$$
where $A\Delta B:=(A\setminus B)\cup (B\setminus A)$ denotes the
symmetric difference of two sets $A$ and $B$.
Suppose that $\S:=\{S_g: g\in G\}$ is a bounded representation of $G$ on $X$.
Then  $(A_\a^\S)_{\a\in \A}$ given by
$$A_\a^\S x:=\frac{1}{|F_\a|}\int_{F_\a} S_g x\, dg\en (x\in X)$$
is a strong right $\S$-ergodic net.
\end{enumerate}
\end{examples}

If $G$ is a right amenable group in the situation of Example \ref{ex:cesaro-nets}
(\ref{ex:folner}), then there always exists a F\o lner net
$(F_\a)_{\a\in \A}$ in $G$ (see \cite[Theorem 4.10]{paterson88}).
Hence, in right amenable groups there always exist strong right $\S$-ergodic nets
for each representation $\S$.  
In \cite [Proposition 1]{sato78}, Sato showed the existence of weak
right (left) ergodic nets in right (left) amenable operator semigroups. 
We give a proof for the case of an amenable semigroup. The one-sided
case is analogous.

  \begin{prop}
    \label{prop:amenable-weak-ergodic-nets}
Let $G$ be represented on $X$ by a bounded (right) amenable
semigroup $\S=\{S_g:g\in G\}$. Then there exists a weak (right)
$\S$-ergodic net in $\L(X)$.
  \end{prop}
  \begin{proof}
    Let $m\in C_b(\S)'$ be an invariant mean.
    Denote by $B$ the
    closed unit ball of  $C_b(\S)'$ and by $\ex B$ the set of extremal
    points of $B$. Since $m$ is a mean, we have $m\in
    B=\ol{\co}\ex B$ by the Krein-Milman theorem, where the closure is
    taken with respect to the weak$^*$-topology. Since $\ex
    B=\{\delta_{S_g}: g\in G\}$, this implies that 
    there exists a net $(\sum_{i=1}^{N_\a}
    \lambda_{i,\a}\delta_{S_{g_i}})_{\a\in\A}\subset \co
    \{\delta_{S_g}: g\in G\}$ with $\lim_\a
    \sum_{i=1}^{N_\a} \lambda_{i,\a}\delta_{S_{g_i}}=m$ in the weak$^*$-topology. 
   Since $m$ is invariant, we obtain
    $$\lim_\a \sum_{i=1}^{N_\a}
    \lambda_{i,\a}\delta_{S_{g_i}}(f-R_{S_g}f)=\lim_\a \sum_{i=1}^{N_\a}
    \lambda_{i,\a}\delta_{S_{g_i}}(f-L_{S_g}f) =0\quad\forall  g\in G, f\in C_b(\S).$$
   Define the net $(A_\a^{\S})_{\a\in\A}$ by
$ A_\a^{\S}:= \sum_{i=1}^{N_\a} \lambda_{i,\a} S_{g_i}\in \co\S$
for $\a\in\A$. To see that $(A_\a^{\S})_{\a\in\A}$ is weakly
asymptotically invariant, let $x\in X$ and $x'\in X'$ and define
$f_{x,x'}\in C_b(\S)$ by $f_{x,x'}(S_g):=\sk{S_g x,x'}$ for
$g\in G$. Then for all $g\in G$ we have 
\begin{align*}
  \sk{A_\a^{\S}x- A_\a^{\S} S_gx,x'}&= \sum_{i=1}^{N_\a}
  \lambda_{i,\a}\left(\sk{S_{g_i}x,x'}-\sk{S_{g_i}S_gx,x'}\right)  \\
&=\sum_{i=1}^{N_\a} \lambda_{i,\a}(f_{x,x'}(S_{g_i})-R_{S_g}f_{x,x'}(S_{g_i}))\\
&=\sum_{i=1}^{N_\a}
\lambda_{i,\a}\delta_{S_{g_i}}(f_{x,x'}-R_{S_g}f_{x,x'})\longrightarrow
0
\end{align*}
and 
$$\sk{A_\a^{\S}x- S_g A_\a^{\S}x,x'}\longrightarrow 0$$
analogously. Hence $(A_\a^{\S})_{\a\in\A}$ is a weak $\S$-ergodic net.
 \end{proof}

We now show that the existence of weak ergodic nets
actually implies the existence of strong ergodic nets. 

  \begin{thm}
\label{thm:existence-of-strong-ergodic-nets}
    Let $G$ be represented on $X$ by a bounded semigroup $\S=\{S_g:g\in G\}$. Then the following assertions are equivalent.
    \begin{enumerate}
    \item There exists a weak (right) $\S$-ergodic net.
    \item There exists a strong (right) $\S$-ergodic net.
    \end{enumerate}
  \end{thm}
 \begin{proof}
We give a proof for the case of an $\S$-ergodic
net. The one-sided case can be shown in a similar way.

(1)$\imp$(2): Consider the locally convex space $E:=\prod_{(g,x)\in
  G\times X}X\times X$ endowed with the product topology, where
$X\times X$ carries the product (norm-)topology. 
Define the linear map
$$\Phi: \L(X)\to E,\quad \Phi(T)=(TS_g x-Tx,S_gTx-Tx)_{(g,x)\in G\times X}.$$
By 17.13(iii) in \cite{kelley-namioka63} the weak topology
$\sigma(E,E')$ on the product $E$ coincides with the product of the
weak topologies of the coordinate spaces. Hence, if $(A_\a^\S)_{\a\in
  \A}$ is a weak $\S$-ergodic net on $X$, then $\Phi(A_\a^\S)\to 0$ with respect to the weak topology on $E$ and thus $0\in \ol{\Phi(\ol{\co} \S)}^{\sigma(E,E')}$. 
Since the weak and strong closure coincide on the convex set $\Phi(\ol{\co} \S)$, there exists a net $(B_\beta^\S)_{\b\in\B}\subset\ol{\co} \S$ with $\Phi(B_\b^\S)\to 0$ in the topology of $E$. By the definition of this topology this means $\|B_\b^\S S_gx-B_\b^\S x\|\to 0$ and $\|S_g B_\b^\S x-B_\b^\S x\|\to 0$ for all $(g,x)\in G\times X$ and hence $(B_\b^\S)_{\b\in\B}$ is a strong $\S$-ergodic net.

(2)$\imp$(1) is clear.
 \end{proof}

The following corollary is a direct consequence of Proposition \ref{prop:amenable-weak-ergodic-nets} and Theorem \ref{thm:existence-of-strong-ergodic-nets}.

 \begin{cor}
\label{cor:existence-of-strong-ergodic-nets}
    Let $G$ be represented on $X$ by a bounded (right)
    amenable semigroup $\S=\{S_g:g\in G\}$. Then there exists a strong
    (right) $\S$-ergodic net.
 \end{cor}

If ergodic nets converge we are led to
the concept of \emph{mean ergodicity}. We use the following abstract notion.

\begin{definition}
 The semigroup $\S$ is called \emph{mean ergodic} if $\ol{\co}\S$
 contains a zero element $P$ (cf. \cite[Chapter 1.1]{berglund89}), called the \emph{mean ergodic projection of $\S$}. 

Notice that for $P$ being a zero element of $\ol{\co}\S$ it suffices that
 $PS_g=S_g P=P$ for all $g\in G$.
\end{definition}
Nagel \cite{nagel73} and Sato \cite{sato78} studied such
semigroups and their results are summarized in Krengel \cite[Chapter
2]{krengel85}.

In the next theorem we collect a series of properties equivalent to
mean ergodicity. Most of them can be found in Krengel \cite[Chapter 2, Theorem
1.9]{krengel85}, but we give a proof for completeness.

Denote the fixed spaces of
$\S$ and $\S'$ by $\Fix\S=\{x\in X: S_g x=x\;\forall g\in G\}$ and
$\Fix\S'=\{x'\in X': S_g' x'=x'\;\forall g\in G\}$ respectively and the linear span
of the set $\rg(I-\S)=\{x-S_g x: x\in X, g\in G\}$ by $\lin\rg(I-\S)$.

\begin{thm}
  \label{thm:mean-ergodic-everywhere}
Let $G$ be represented on $X$ by a bounded right amenable
semigroup $\S=\{S_g:g\in G\}$. Then the following assertions are equivalent.
\begin{enumerate}
\item $\S$ is mean ergodic with mean ergodic projection $P$.
\item $\ol{\co}\S x\cap \Fix\S\neq\emptyset$ for all $x\in X$.
\item $\Fix\S$ separates $\Fix\S'$.
\item $X= \Fix\S\oplus \ol{\lin}\rg(I-\S)$.   
\item $A_\a^\S x$ has a weak cluster point in $\Fix\S$ for some/every
  weak right $\S$-ergodic net $(A_\a^\S)$ and all $x\in X$.
\item $A_\a^\S x$ converges weakly to a fixed point of $\S$ for
  some/every weak right $\S$-ergodic net~$(A_\a^\S)$ and all $x\in X$.
\item $A_\a^\S x$ converges weakly to a fixed point of $\S$ for
  some/every strong right $\S$-ergodic net $(A_\a^\S)$ and all $x\in X$.
\item $A_\a^\S x$ converges strongly to a fixed point of $\S$ for
  some/every strong right $\S$-ergodic net $(A_\a^\S)$ and all $x\in X$.
\end{enumerate}
The limit $P$ of the nets $(A_\a^\S)$ in the weak and strong operator
topology, respectively, is the mean ergodic projection of $\S$ mapping $X$ onto
$\Fix \S$ along $\ol{\lin}\rg(I-\S)$. 
\end{thm}

\begin{proof}
(1)$\imp$(2): Since $S_gP=P$ for all $g\in G$, we have $Px\in\ol{\co}\S x\cap \Fix\S$
for all $x\in X$.

(2)$\imp$(3): Let $0\neq x'\in \Fix\S'$. Take $x\in X$ such that
$\sk{x',x}\neq 0$. If $y\in \ol{\co}\S x\cap \Fix\S$ then we have
$\sk{x',y}=\sk{x',x}\neq 0$. Hence $\Fix\S$ separates $\Fix\S'$.

(3)$\imp$(4): 
Let $(A_\a^\S)$ be a weak right $\S$-ergodic net and let $x'\in X'$ vanish
on $\Fix\S\oplus \ol{\lin}\rg(I-\S)$. 
Then in particular $\sk{x',y}=\sk{x',S_g y}=\sk{S_g'x',y}$ for
all $y\in X$ and $g\in G$. Hence $x'\in \Fix\S'$. Since $\Fix\S$
separates $\Fix\S'$ and $x'$ vanishes on $\Fix\S$, this implies
$x'=0$. 
Hence $\Fix\S\oplus \ol{\lin}\rg(I-\S)$ is dense in $X$ by the
Hahn-Banach theorem and it remains to show that $\Fix\S\oplus
\ol{\lin}\rg(I-\S)$ is closed. 
For $D:=\{x\in X: \sigma\text{-}\lim_\a A_\a^\S x\text{
  exists}\}$ we obtain $D=\Fix\S\oplus \ol{\lin}\rg(I-\S)$ and $D$ is
closed since $(A_\a^\S)$ is uniformly bounded.

(4)$\imp$(6): Let $(A_\a^\S)$ be any weak right $\S$-ergodic net. Then
$A_\a^\S x$ converges weakly to a fixed point of $\S$ for all $x\in
X$. 
Indeed, the convergence on
$\Fix\S$ is clear and the weak convergence to $0$ on $\lin\rg(I-\S)$
follows from weak right asymptotic
invariance and linearity of $(A_\a^\S)$. 
Since the set $\{x\in X:
\sigma\text{-}\lim_\a A_\a^\S x\text{ exists}\}$ is closed we obtain
weak convergence on all of $\Fix \S \oplus \ol{\lin}\rg(I-\S)$. The
limit of the net $(A_\a^\S)$ is the projection onto $\Fix\S$ along $\ol{\lin}\rg(I-\S)$.

(4)$\imp$(8): An analogous reasoning as in (4)$\imp$(6) yields the
strong convergence of $A_\a^\S x$ for every strong right $\S$-ergodic
net $(A_\a^\S)$ and every $x\in X$. 

(5)$\imp$(2): 
Let $(A_\a^\S)$ be a weak right $\S$-ergodic net, take $x\in X$ and define $Px$ as the weak
limit of a convergent subnet $(A_{\b_x}^\S x)$ of $(A_\a^\S x)$. 
Then $Px\in\ol{\co}\S x\cap\Fix\S$ for all $x\in X$.

(6)$\imp$(1): 
Let $(A_\a^\S)$ be a weak right $\S$-ergodic net. Defining $Px$ as the weak
limit of $A_\a^\S x$ for each $x\in X$ we obtain $Px\in\ol{\co}\S x$ for
all $x\in X$.
Furthermore, for all $x\in X$ and $g\in G$ we
  have $Px-PS_gx=\sigma\text{-}\lim_{\a} A_{\a}^\S x-A_{\a}^\S S_gx=0$ and
  $Px-S_gPx$ since $Px\in \Fix\S$. 
Hence $S_gP=PS_g=P$ for all $g\in G$ and \cite[Theorem 1.2]{nagel73}
yield the mean ergodicity of $\S$.

The remaining implications are trivial.
\end{proof}

If the semigroup $\S$ is amenable and $(A_\a^\S)$ is a 
convergent $\S$-ergodic 
net, then the limit $Px:=\lim_\a A_\a^\S x$ is automatically a
fixed point of $\S$ for each $x\in X$, since $S_g Px-Px=\lim_{\a} A_{\a}^\S
x-S_g A_{\a}^\S x=0$ for each $g\in G$ by the asymptotic left
invariance. Hence the following corollary is a direct consequence of
Theorem \ref{thm:mean-ergodic-everywhere}.

\begin{cor}
  \label{cor:mean-ergodic-everywhere}
Let $G$ be represented on $X$ by a bounded amenable
semigroup $\S=\{S_g:g\in G\}$. Then the following assertions are equivalent.
\begin{enumerate}
\item $\S$ is mean ergodic with mean ergodic projection $P$.
\item $\ol{\co}\S x\cap \Fix\S\neq\emptyset$ for all $x\in X$.
\item $\Fix\S$ separates $\Fix\S'$.
\item $X= \Fix\S\oplus \ol{\lin}\rg(I-\S)$.   
\item $A_\a^\S x$ has a weak cluster point for some/every
  weak $\S$-ergodic net $(A_\a^\S)$ and all $x\in X$.
\item $A_\a^\S x$ converges weakly for
  some/every weak $\S$-ergodic net~$(A_\a^\S)$ and all $x\in X$.
\item $A_\a^\S x$ converges weakly for
  some/every strong $\S$-ergodic net $(A_\a^\S)$ and all $x\in X$.
\item $A_\a^\S x$ converges strongly for
  some/every strong $\S$-ergodic net $(A_\a^\S)$ and all $x\in X$.
\end{enumerate}
The limit $P$ of the nets $(A_\a^\S)$ in the weak and strong operator
topology, respectively, is the mean ergodic projection of $\S$ mapping $X$ onto
$\Fix \S$ along $\ol{\lin}\rg(I-\S)$. 
\end{cor}

The next result can be found in Nagel \cite[Satz 1.8]{nagel73} (see
also Ghaffari \cite[Theorem 1]{ghaffari07}). 
We give a different proof.
\begin{cor}
\label{cor:amenable_mean_ergodic}
Let $G$ be represented on $X$ by a bounded amenable
semigroup $\S=\{S_g:g\in G\}$. If $\S$ is relatively compact with respect to
  the weak operator topology, then $\S$ is mean ergodic.
\end{cor}
\begin{proof}
  Since $\S x$ is relatively weakly compact, we obtain that $\ol{\co}\S x$ is weakly compact for all $x\in X$ by the Krein-\v{S}mulian Theorem. Hence, if $(A_\a^\S)$ is a weak $\S$-ergodic net, then $A_\a^\S x$ has a weak cluster point for each $x\in X$. The mean ergodicity of $\S$ then follows from Corollary~\ref{cor:mean-ergodic-everywhere}.
\end{proof}

If the Banach space satisfies additional geometric properties,
contractivity of the semigroup implies amenability and mean
ergodicity.  
For uniformly convex spaces with strictly convex dual unit balls this has been shown by Alaoglu and Birkhoff \cite[Theorem 6]{alaoglu-birkhoff40} using the so-called \emph{minimal method}.
In \cite[Theorem 6']{day41} Day observed that the same method still works if uniform convexity is replaced by strict convexity.

\begin{cor}
  \label{cor:von-Neumann}
Let $X$ be a reflexive Banach space such that the unit balls of $X$
and $X'$ are strictly convex. If the semigroup
$G$ is represented on $X$ by a semigroup of contractions
$\S=\{S_g:g\in G\}$, then $\S$ is mean ergodic.
\end{cor}
\begin{proof}
If $\S$ is a contractive semigroup in $\L(X)$ and the unit balls of $X$
and of $X'$ are strictly convex, then $\S$ is
amenable by \cite[Corollary 4.14]{deleeuw-glicksberg61}. 
Since $\S$ is
bounded on the reflexive space $X$, it follows that $\S$ is relatively compact with respect
to the weak operator topology. 
Hence Corollary \ref{cor:amenable_mean_ergodic} implies that $\S$ is mean ergodic.
\end{proof}

In some situations (see e.g. Assani \cite[Theorem 2.10]{assani03},
Walters \cite[Theorem 4]{walters96},
Lenz \cite[Theorem 1]{lenz09}) 
one is interested in convergence of an ergodic net only on
some given $x\in X$. 
Apart from the implication (2)$\imp$(3) the following result is a direct consequence of
Theorem~\ref{thm:mean-ergodic-everywhere} and
Corollary~\ref{cor:mean-ergodic-everywhere}  applied to the 
restriction $\S|_{Y_x}$ of $\S$ to the closed $\S$-invariant subspace $Y_x:=\ol{\lin}\S x$.

\begin{prop}
  \label{prop:mean-ergodic-for-x}
Let $G$ be represented on $X$ by a bounded (right) amenable
semigroup $\S=\{S_g:g\in G\}$ and let $x\in X$. Then the following
assertions are equivalent.

\begin{enumerate}
\item $\S$ is mean ergodic on $Y_x$ with mean ergodic projection $P$.
\item $\ol{\co}\S x\cap \Fix\S\neq\emptyset$.
\item $\Fix\S|_{Y_x}$ separates $\Fix\S|_{Y_x}'$.
\item $x\in \Fix\S\oplus\ol{\lin}\rg(I-\S)$.
\item $A_\a^\S x$ has a weak cluster point (in $\Fix\S$) for
  some/every weak (right) $\S$-ergodic net $(A_\a^\S)$.
\item $A_\a^\S x$ converges weakly (to a fixed point of $\S$) for
  some/every weak  (right) $\S$-ergodic net
  $(A_\a^\S)$.
\item $A_\a^\S x$ converges weakly (to a fixed point of $\S$) for
  some/every strong (right) $\S$-ergodic net
  $(A_\a^\S)$.
\item $A_\a^\S x$ converges strongly (to a fixed point of $\S$) for
  some/every strong  (right) $\S$-ergodic net $(A_\a^\S)$.
\end{enumerate}
The limit $P$ of the nets $A_\a^\S$ in the weak and strong operator
topology on $Y_x$, respectively, is the mean ergodic projection of
$\S|_{Y_x}$ mapping $Y_x$ onto $\Fix\S|_{Y_x}$ along $\ol{\lin}\rg(I-\S|_{Y_x})$.
\end{prop}
\begin{proof}
  (2)$\imp$(3): Let $0\neq x'\in \Fix\S|_{Y_x}'$ and take $y\in Y_x$
  such that $\sk{x',y}\neq 0$. Since $x$ generates the space $Y_x$
  this yields $\sk{x',x}\neq 0$. If $z\in \ol{\co}\S x\cap \Fix\S$,
  then $z\in Y_x$ and we have $\sk{x',z}=\sk{x',x}\neq 0$. Hence
  $\Fix\S|_{Y_x}$ separates $\Fix\S|_{Y_x}'$.

The other implications follow directly from
Theorem~\ref{thm:mean-ergodic-everywhere} and
Corollary~\ref{cor:mean-ergodic-everywhere} by noticing that the set
$\{y\in X: \lim_\a A_\a^\S y \text{ exists}\}$ is a closed $\S$-invariant
subspace of $X$.
\end{proof}

\section{Uniform families of ergodic operator nets}

We now use the above results on mean ergodic semigroups in order to
study the convergence of uniform families of ergodic nets.

Let $I$ be an index set and suppose that the semigroup $G$
is represented on $X$ by bounded semigroups
$\S_i=\{S_{i,g}:g\in G\}$ for each $i\in I$. 
Moreover, we assume that the $\S_i$ are \emph{uniformly bounded}, i.e., 
$\sup_{i\in I}\sup_{g\in G}\|S_{i,g}\|<\infty$. 
\begin{definition}
\label{definition:uniformly-family}
Let $\A$ be a directed set and let $(A_\a^{\S_i})_{\a\in\A}\subset
\L(X)$ be a net of operators for each $i\in I$. 
Then $\{(A_\a^{\S_i})_{\a\in\A}: i\in I\}$ is a \emph{uniform family
  of right (left) ergodic nets} if
\begin{enumerate}
\item  $\forall\a\in\A, \forall \e>0, \forall x_1,\dots, x_m\in X,
  \exists g_1,\dots, g_N\in G$ such that for each $i\in I$
  there exists a convex combination $\sum_{j=1}^{N}c_{i,j}
  S_{i,g_j}\in \co\S_i$ 
  satisfying
$$\sup_{i\in I}\|A_\a^{\S_i}x_k-\textstyle{\sum_{j=1}^{N}}c_{i,j}
S_{i,g_j}x_k\|<\e\quad \forall k\in\{1,\dots, m\};$$
\item $\displaystyle \lim_\a \sup_{i\in
    I}\|A_\a^{\S_i}x-A_\a^{\S_i}S_{i,g}x\|= 0 \quad \left(\lim_\a \sup_{i\in
    I}\|A_\a^{\S_i}x-S_{i,g}A_\a^{\S_i}x\|=0\right) \quad \forall g\in G,
   x\in X.$
\end{enumerate}
The set $\{(A_\a^{\S_i})_{\a\in\A}: i\in I\}$ is called a \emph{uniform family of
ergodic nets} if it is a uniform family of
left and right ergodic nets.
\end{definition}
Notice that if 
$\{(A_\a^{\S_i})_{\a\in\A}: i\in I\}$ is a uniform family of (right)
ergodic nets, then each $ (A_\a^{\S_{i}})_{\a\in\A}$ is a 
strong (right) $\S_i$-ergodic net.

Here are some examples of uniform families of ergodic nets.
\begin{prop}
\label{prop:uniformly-ergodic}
  \begin{enumerate}[(a)]

  \item          
Let $S\in \L(X)$ with $\|S\|\le 1$. Consider the semigroup $(\N,+)$ being represented on $X$ by the families $\S_\lambda=\{(\lambda S)^n: n\in \N\}$ for $\lambda\in\Torus$.
Then 
$$\textstyle\left\{\left( \frac{1}{N}\sum_{n=0}^{N-1}(\lambda
    S)^n\right)_{N\in\N}: \lambda\in\Torus\right\}$$
is a uniform family of ergodic nets.

\item         
 In the situation of (a),
$$\textstyle\left\{\left( (1-r)\sum_{n=0}^{\infty}r^n\lambda^n
    S^n\right)_{0<r<1}: \lambda\in\Torus\right\}$$
is a uniform family of ergodic nets.

\item         
Let $K$ be a compact space and $\ph:K\to K$ a continuous
  transformation. Let $H$ be a Hilbert space and $S: f\mapsto f\circ\ph$ the Koopman
  operator corresponding to $\ph$ on the space $C(K,H)$ of continuous $H$-valued functions on $K$. Denote by $U(H)$ the set of unitary operators on
  $H$ and by $\Lambda$ the set of continuous maps $\gamma:K\to
  U(H)$. Consider the semigroup $(\N,+)$ and its representations
  on $C(K,H)$  given by the families $\S_\gamma=\{(\gamma
  S)^n: n\in\N\}$ for $\gamma\in\Lambda$, where $(\gamma S)
  f(x)=\gamma(x)Sf(x)$ for $x\in K$ and $f\in C(K,H)$. Then 
$$\textstyle\left\{\left( \frac{1}{N}\sum_{n=0}^{N-1}(\gamma
    S)^n\right)_{N\in\N}: \gamma\in\Lambda\right\}$$
is a uniform family of ergodic nets.

\item          
Let $(S(t))_{t\in\R_+}$ be a bounded $C_0$-semigroup on $X$. Consider
  the semigroup $(\R_+,+)$ being represented on $X$ by the families
  $\S_r=\{e^{2\pi i rt}S(t): t\in\R_+ \}$ for $r\in B$,
  where $B\subset\R$ is bounded. Then
$$\textstyle\left\{\left ( \frac{1}{s}\int_0^s e^{2\pi i r t} S(t)\,dt\right)_{s\in\R_+}:
  r\in B \right\}$$
is a uniform family of ergodic nets.

\item          
Let $\S=\{S_g:g\in G\}$ be a bounded representation on $X$ of an abelian semigroup $G$. Order the elements of $\co\S$ by
  setting $U\le V$ if there exists $W\in \co\S$ such that
  $V=WU$. Denote by $\widehat{G}$ the character semigroup of $G$, i.e., the set of continuous multiplicative maps $G\to\Torus$, and
  consider the representations $\S_\chi=\{\chi(g)S_g: g\in G\}$ of $G$ on $X$
  for $\chi\in\widehat{G}$. Then
$$\textstyle\left\{\left ( \sum_{i=1}^N c_i \chi(g_i)S_{g_i}
\right )_{\sum_{i=1}^N c_i S_{g_i}\in\co\S}: \chi\in \widehat{G}\right\}$$
is a uniform family of ergodic nets.

\item              
\label{item:folner}
 Let $H$ be a locally compact group with left Haar measure
  $|\cdot|$ and let $G\subset H$ be a subsemigroup. Suppose that there
  exists a \emph{F\o lner net} $(F_\a)_{\a\in \A}$ in $G$. 
 Suppose that $\S:=\{S_g: g\in G\}$ is a bounded representation of $G$
 on $X$. Consider the representations $\S_\chi=\{\chi(g)S_g: g\in G\}$
 of $G$ on $X$ for $\chi\in \Lambda$, where $\Lambda\subset\widehat{H}$ is uniformly
  equicontinuous on compact sets. 
Then
$$\textstyle\left\{\left ( \frac{1}{|F_\a|}\int_{F_\a} \chi(g) S_g \,dg\right)_{\a\in \A}:
  \chi\in \Lambda\right\}$$
is a uniform family of right ergodic nets.

  \end{enumerate}
\end{prop}

\begin{proof}

  \begin{enumerate}[(a)]
  \item  Property (1) of Definition \ref{definition:uniformly-family}
    is clear. To see (2), let $x\in X$ and $k\in\N$. Then 
 \begin{align*}
  \sup_{\lambda\in\Torus}\left\|\frac{1}{N}\sum_{n=0}^{N-1}(\l S)^nx-(\l
  S)^{n+k}x\right\|&\le \sup_{\l\in\Torus} \frac{1}{N}\sum_{n=0}^{k-1}\|(\l S)^n
x\|+\frac{1}{N}\sum_{n=N}^{N-1+k}\|(\l S)^n x\|\\
& \le \frac{2k}{N}\|x\|\xrightarrow[N\to\infty]{}0.
 \end{align*}

\item (1): Let $0<r<1$ and $\e>0$. Choose $N\in\N$ such that
$r^{N}<\frac{\e}{2}$. Then for all $\l\in \Torus$ we have
\begin{align*}
  \|(1-r)& \sum_{n=0}^{\infty}r^n\lambda^n
    S^n-\frac{(1-r)}{(1-r^N)} \sum_{n=0}^{N-1}r^n\lambda^n
    S^n \|\\
&\le \left\|(1-r) \sum_{n=0}^{\infty}r^n\lambda^n
    S^n-(1-r) \sum_{n=0}^{N-1}r^n\lambda^n
    S^n\right\|+ \\
&\hspace{4cm} \left \|(1-r) \sum_{n=0}^{N-1}r^n\lambda^n
    S^n - \frac{(1-r)}{(1-r^{N})} \sum_{n=0}^{N-1}r^n\lambda^n
    S^n \right\|\\
&\le (1-r)\sum_{n=N}^\infty r^n +
(1-r)\left|1-\frac{1}{(1-r^N)}\right| \sum_{n=0}^{N-1} r^n \\
&\le r^N +r^N < \e.
\end{align*}

(2): Let $x\in X$ and $k\in\N$. Define for each $\l\in\Torus$ the
sequence $x^{(\l)}$ by $x_n^{(\l)}:=(\l S)^nx-(\l S)^{n+k}x$ for
$n\in\N$. Then it follows from (a) that
$\sup_{\lambda}\|\frac{1}{N}\sum_{n=0}^{N-1} x_n^{(\l)}\|\longrightarrow
0$.  It is well known that Cesàro convergence implies the
convergence of the
Abel means to the same limit (see \cite[Proposition
2.3]{li07}). One checks  
that if the Cesàro convergence is
uniform in $\l\in\Torus$, then the convergence of the Abel means is also
uniform. Hence we obtain $\lim_{r\uparrow 1}\sup_{\l\in\Torus}\|(1-r)\sum_{n=0}^\infty r^n
x_n^{(\l)}\|= 0$.

\item 
(1) is clear. To see (2) let $f\in C(K,H)$ and $k\in\N$. Then 
$$\|(\gamma S) f\|=\sup_{x\in K}\|\gamma(x)f(\ph(x))\|_H =\sup_{x\in
  K}\|f(\ph(x))\|_H \le \|f\|,$$ since $\gamma(x)$ is unitary for all
$x\in K$. Hence
\begin{align*}
  \sup_{\gamma\in\Lambda}\left\|\frac{1}{N}\sum_{n=0}^{N-1}(\gamma S)^nf-(\gamma
  S)^{n+k}f\right\|&\le \sup_{\gamma\in\Lambda} \frac{1}{N}\sum_{n=0}^{k-1}\|(\gamma S)^n
f\|+\frac{1}{N}\sum_{n=N}^{N-1+k}\|(\gamma S)^n f\|\\
& \le \frac{2k}{N}\|f\|\xrightarrow[N\to\infty]{}0.
 \end{align*}

\item 
This is a special case of (\ref{item:folner}) for the F\o lner net
$([0,s])_{s>0}$ in $\R_+$ and the set $\Lambda:=\{\chi_r: \R_+ \to
\Torus: r\in B\}$, where $\chi_r(t)=e^{2\pi i r t}$ for $t\in \R_+$. Notice that
$\Lambda\subset \widehat{\R}$ is uniformly equicontinuous on compact sets since $B$ is bounded.

\item (1) is clear. To see (2) let $x\in X$, $g\in G$ and $\e>0$. Choose $N\in\N$ such that
$\frac{2M^2\|x\|}{N}<\e$, where $M=\sup_{g\in G}\|S_g\|$. Then for all
$V=W \frac{1}{N}\sum_{n=0}^{N-1}S_g^n \ge
\frac{1}{N}\sum_{n=0}^{N-1}S_g^n$, where $W= \sum_{i=0}^{k-1}c_i
S_{g_i} \in\co\S$, we have
\begin{align*}
  \sup_{\chi\in\widehat{G}}\left\|\sum_{i=0}^{k-1} \right.&\left . \sum_{n=0}^{N-1}\frac{1}{N}c_i \chi(g_i   g^n)S_{g_ig^n}x-\sum_{i=0}^{k-1}\sum_{n=0}^{N-1}\frac{1}{N}c_i
    \chi(g_i g^n)S_{g_ig^n}\chi(g)S_g x\right\|\\
&\le \sup_{\chi\in\widehat{G}}\left\| \sum_{i=0}^{k-1}c_i\chi(g_i)
S_{g_i}\right\|\cdot  \left\| \frac{1}{N}\sum_{n=0}^{N-1}(\chi(g)S_g)^n(x-\chi(g) S_g x)\right\|
\\
&\le \sup_{\chi\in\widehat{G}} M \frac{1}{N}\|x-(\chi(g)S_g)^N x\|\le M \frac{1}{N}2M\|x\|<\e.
\end{align*}

\item     
(1): Let $\a\in \A$, $\e>0$ and $x_1,\dots, x_m\in X$. Since $F_\a$ is
compact and $\Lambda$ is uniformly
  equicontinuous on $F_\a$ the family $\{g\mapsto \chi(g)S_gx_k:
  \chi\in\Lambda\}$ is also uniformly equicontinuous on $F_\a$ for
  each $k\in\{1,\dots, m\}$. Hence for each $k\in\{1,\dots, m\}$ we can
choose an open neighbourhood $U_k$ of the unity of $H$
  satisfying 
$$g,h\in G, \en h^{-1}g\in U_k \en \imp\en \sup_{\chi\in
  \Lambda}\|\chi(g)S_g x_k-\chi(h)S_g x_k\|<\e.$$
Then $U:=\bigcap_{k=1}^m U_k$ is still an open neighbourhood of
unity. Since $F_\a$ is compact there exists $g_1,\dots, g_N\in F_\a$
such that $F_\a\subset\bigcup_{n=1}^N g_n U$. Defining $V_1:=g_1 U\cap
F_\a$ and $V_n:=(g_n U\cap F_\a) \setminus V_{n-1}$ for $n=2,\dots, N$
we obtain a disjoint union $F_\a=\bigcup_{n=1}^N V_n.$
Hence for all $\chi\in \Lambda$ and $k\in \{1,\dots, m\}$ we have
\begin{align*}
  \left\| \frac{1}{|F_\a|}\int_{F_\a}\chi(g)S_g x_k
    dg-\right . &\left .\sum_{n=1}^N \frac{|V_n|}{|F_\a|} \chi(g_n) S_{g_n}x_k
  \right\|\\
&\le \frac{1}{|F_\a|}\sum_{n=1}^N
  \int_{V_n}\underbrace{\|\chi(g)S_g x_k-\chi(g_n)S_{g_n} x_k\|}_{<\e}dg\\
&< \frac{1}{|F_\a|}\sum_{n=1}^N |V_n|\e=\e.
\end{align*}

(2): If $x\in X$ and $h\in G$, then we have 
\begin{align*}
\sup_{\chi\in\Lambda}\left\|\frac{1}{|F_\a|}\int_{F_\a}\chi(g)S_gx-\chi(hg)S_{hg}x\,dg\right\|
&
  \le  \sup_{\chi\in\Lambda} \frac{1}{|F_\a|}\int_{F_\a\triangle h
    F_\a} \|\chi(g) S_g x\|\, dg\\
&\le \frac{|F_\a\triangle h F_\a|}{|F_\a|} \sup_{g\in G}\|S_g\| \|x\|\longrightarrow 0.
\end{align*}
\end{enumerate}
\end{proof}

Now, let $\{(A_\a^{\S_i})_{\a\in\A}: i\in I\}$ be a uniform family of
right ergodic nets. 
If $x\in X$ and $\S_i$ is right amenable and mean ergodic on $\ol{\lin}\S_i x$ for each
$i\in I$, then it follows from Proposition
\ref{prop:mean-ergodic-for-x} that $A_\a^{\S_i}x\to P_i
x$ for each $i\in I$, where $P_i$ denotes the mean ergodic
projection of $\S_i|_{\ol{\lin}\S_i x}$.
The question arises, when this convergence is uniform in $i\in I$. 
The following elementary example shows that in general this cannot be
expected. 

\begin{example}
  Let $X=\C$ and let $S=I_\C\in L(\C)$ be the identity operator on
  $\C$. Consider the semigroup $(\N,+)$ being represented on $\C$ by the
  families $\S_\lambda=\{\lambda^n I_\C: n\in\N\}$ for
  $\lambda\in\Torus$. Then for each $\lambda\in \Torus$ the
  Cesàro means $\frac{1}{N}\sum_{n=0}^{N-1}\lambda^n$ converge, but
  the convergence is not uniform in $\l\in\Torus$.
\end{example}
However, the following theorem 
gives a sufficient condition for uniform convergence.

\begin{thm}
  \label{thm:uniform-convergence-Banach-space}
Let $I$ be a compact set and let $G$ be represented on $X$ by the
uniformly bounded right amenable semigroups $\S_i=\{S_{i,g}: g\in G\}$ for all $i\in
I$. Let $\{(A_\a^{\S_i})_{\a\in\A}: i\in I\}$ be a uniform family of
right ergodic nets.  Take $x\in X$ and assume that
\begin{enumerate}[(a)]
\item 
$\S_i$ is mean ergodic on $\ol{\lin}\S_i x$ with mean ergodic projection $P_i$ for each $i\in I$, 
\item $I\to \R_+$, $i\mapsto \|A_\a^{\S_i}x-P_ix\|$ is continuous for each $\a\in\A$.
\end{enumerate}
Then 
$$\lim_\a\sup_{i\in I}\|A_\a^{\S_i}x-P_ix\|= 0.$$
\end{thm}

For the proof of Theorem  \ref{thm:uniform-convergence-Banach-space} we need a lemma.

\begin{lemma}
  \label{lemma:uniform-cesaro}
If $\{(A_\a^{\S_i})_{\a\in\A}: i\in I\}$ is a uniform family of right
ergodic nets, then 
$$\lim_\a\sup_{i\in
  I}\|A_\a^{\S_i}x-A_\a^{\S_i}A_\b^{\S_i}x\|=0\quad\forall \b\in\A,
x\in X.$$
\end{lemma}
\begin{proof}
  Let $\b\in\A$, $\e>0$ and $x\in X$. 
Since $\{(A_\a^{\S_i})_{\a\in\A}: i\in I\}$ is a uniform family of
right ergodic nets, there exists $g_1,\dots, g_N\in G$
and for each $i\in I$ a
  convex combination $\sum_{j=1}^{N}c_{i,j} S_{i,g_j}\in \co\S_i $
  such that $\sup_{i\in I}\|A_\b^{\S_i}x-\sum_{j=1}^{N}c_{i,j}
S_{i,g_j}x\|<\e/M$, where $M=\sup_{i\in I}\sup_{g\in G}\|S_{i,g}\|$. Now, choose $\a_0\in\A$ such that for all $\a>\a_0$
$$\sup_{i\in I}\|A_\a^{\S_i}x-A_\a^{\S_i}S_{i,g_j}x\|<\e\quad\forall
j=1,\dots,N.$$
Then for all $\a>\a_0$ and $i\in I$ we obtain 
\begin{align*}
\textstyle
  \|A_\a^{\S_i}x-A_\a^{\S_i}A_\b^{\S_i}x\|&\textstyle\le  \|A_\a^{\S_i}x-A_\a^{\S_i}\sum_{j=1}^{N}c_{i,j}
S_{i,g_j}x\|+\|A_\a^{\S_i}\sum_{j=1}^{N}c_{i,j}
S_{i,g_j}x-A_\a^{\S_i}A_\b^{\S_i}x\|\\
&\le\textstyle
\sum_{j=1}^{N}c_{i,j}\|A_\a^{\S_i}x-A_\a^{\S_i}S_{i,g_j}x\|+M\|\sum_{j=1}^{N}c_{i,j}
S_{i,g_j}x-A_\b^{\S_i}x\|\\
&< 2\e.
\end{align*}
\end{proof}

\begin{proof}[Proof of Theorem
  \ref{thm:uniform-convergence-Banach-space}]
By Proposition \ref{prop:mean-ergodic-for-x} and the hypotheses the function $f_\a:I\to\R_+$ defined by 
  $f_\a(i)=\|A_\a^{\S_i}x-P_ix\|$ is continuous for each $\a\in\A$ and
  $\lim_\a f_\a(i)=0$ for all $i\in I$.
%
Compactness and continuity yield a net $(i_\a)_{\a\in \A}\subset I$ with  $\sup_{i\in
    I}f_\a(i)=f_\a(i_\a)$ for all $\a\in \A$. To show that $\lim_\a\sup_{i\in
    I}f_\a(i)=0$ it thus suffices to show that every subnet of
  $(f_\a(i_\a))$ has a subnet converging to $0$. 
Let
  $(f_{\a_k}(i_{\a_k}))$ be a subnet of $(f_\a(i_\a))$ and let $\e>0$. Since $I$ is
  compact, we can choose a subnet of $(i_{\a_k})$, also denoted by $(i_{\a_k})$, such that $i_{\a_k}\to
  i_0$ for some $i_0\in I$. Since $f_\a$ converges pointwise to $0$,
  we can take $\b\in\A$ such that $f_{\b}(i_0)<\e/M$, where
  $M=\sup_{i\in I}\sup_{g\in G}\|S_{i,g}\|$. By continuity of $f_{\b}$ there
  exists $k_1$ such that for all $k>k_1$ we have
  $f_{\b}(i_{\a_k})-f_{\b}(i_0)<\e/M$. By Lemma
\ref{lemma:uniform-cesaro} there exists $k_2>k_1$ such that for all $k>k_2$
\begin{align*}
  f_{\a_k}(i_{\a_k})&\le\|A_{\a_k}^{\S_{i_{\a_k}}}x-A_{\a_k}^{\S_{i_{\a_k}}}A_{\b}^{\S_{i_{\a_k}}}x\|+\|A_{\a_k}^{\S_{i_{\a_k}}}A_{\b}^{\S_{i_{\a_k}}}x-A_{\a_k}^{\S_{i_{\a_k}}}P_{i_{\a_k}}x\|\\
& \le \e +M \|A_{\b}^{\S_{i_{\a_k}}}x-P_{i_{\a_k}}x\| \\
&\le \e +M (f_\b(i_{\a_k})-f_\b(i_0))+ M f_\b(i_0)\\
&\le 3\e.
\end{align*}
Hence $\lim_\a \sup_{i\in I} f_\a(i)=0$.
\end{proof}

We now apply the above theory  to
operator semigroups on the space $C(K)$ of complex valued continuous functions on a
compact metric space $K$.

\begin{cor}
\label{cor:assani}
  Let $\ph:K\to K$ be a continuous
  map, $S:f\mapsto f\circ\ph$ the corresponding Koopman operator on $C(K)$,
 and assume that there exists a
  unique $\ph$-invariant Borel probability measure $\mu$ on $K$. If
  $f\in C(K)$ satisfies $P_\lambda f=0$ for all $\lambda\in\Torus$,
  where $P_\lambda$ denotes the mean ergodic projection of $\{(\lambda
  S)^n: n\in\N\}$ on $L^2(K,\mu)$, then
  \begin{enumerate}
  \item $\displaystyle\lim_{N\to\infty}\sup_{\lambda\in\Torus}\left\|\textstyle\frac{1}{|F_N|}\sum_{n\in F_N}\lambda^nS^n
f\right\|_{\infty}=0$\en
for each F\o lner sequence $(F_N)_{N\in\N}$ in $\N$,
\item $\displaystyle\lim_{r\uparrow 1}\sup_{\lambda\in\Torus}\left\|(1-r)\textstyle\sum_{n=0}^{\infty}r^n\lambda^nS^n f\right\|_\infty=0.$
  \end{enumerate}

\end{cor}
\begin{proof}
 By the hypotheses it follows from Robinson \cite[Theorem
 1.1]{robinson94} and Proposition \ref{prop:mean-ergodic-for-x} that
 the semigroups $\S_{\lambda}=\{(\lambda S)^n: 
  n\in\N\}$ are mean ergodic on $\ol{\lin}\S_{\lambda} f\subset C(K)$ for each
  $\lambda\in\Torus$. 
 Let now $(F_N)_{N\in\N}$ be a F\o lner sequence in $\N$ and consider
 the uniform family of ergodic nets 
  $$\textstyle\left\{\left(\frac{1}{|F_N|}\sum_{n\in F_N}\lambda^nS^n\right)_{N\in\N}:
    \lambda\in\Torus\right\}$$ 
(cf. Proposition
  \ref{prop:uniformly-ergodic} (\ref{item:folner})). If $P_\lambda
  f=0$ for all $\lambda\in \Torus$, then also condition (b) in
  Theorem~\ref{thm:uniform-convergence-Banach-space} is satisfied
  since the map 
$\lambda\mapsto\frac{1}{|F_N|}\sum_{n\in F_N}\lambda^nS^n f$ is
continuous for each $N\in\N$. Hence
$\lim_{N\to\infty}\sup_{\lambda\in\Torus}\|\frac{1}{|F_N|}\sum_{n\in
  F_N} \lambda^nS^n f\|_{\infty}=0$ by Theorem~\ref{thm:uniform-convergence-Banach-space}. 

The same reasoning applied to the uniform family of ergodic nets 
$$\textstyle\left\{\left((1-r)\sum_{n=0}^{\infty} r^n \lambda^n
    S^n\right)_{0<r<1}: \lambda\in\Torus\right\}$$
 yields the second assertion.
\end{proof}

\begin{remark}
  In \cite[Theorem 2.10]{assani03} Assani proved the first assertion of Corollary~\ref{cor:assani} for the F\o lner sequence $F_N=\{0,\dots,N-1\}$ in $\N$. 
Generalisations of this result can be found in Walters
\cite[Theorem~5]{walters96}, Santos and Walkden
\cite[Prop.~4.3]{santos07} and Lenz \cite[Theorem~2]{lenz09}. We will
systematically study and unify these cases in a subsequent paper. 
\end{remark}

\textbf{Acknowledgement.} The author is grateful to Rainer Nagel for his support, valuable discussions and comments.

\bibliographystyle{siam}
 \bibliographystyle{amsplain}

\end{document}